\documentclass{amsart}
\usepackage{amsthm}
\usepackage{amssymb}
\usepackage{amsmath}
\usepackage{mathrsfs}
\usepackage{euscript}
\newtheorem{theorem}{Theorem}[section]
\newtheorem{definition}{Definition}[section]

\newtheorem{lemma}{Lemma}[section]
\newtheorem{corollary}{Corollary}[section]
\numberwithin{equation}{section}

\begin{document}
\title{Change of basis for generalized Chebyshev Koornwinder's type polynomials of first kind}
\author{Mohammad A. AlQudah}
\address{School of Basic Sciences and Humanities, German Jordanian University, Amman 11180, Jordan}
\email{mohammad.qudah@gju.edu.jo}
\date{\today}
\author{Maalee N. AlMheidat}
\address{Department of Basic Sciences, University of Petra, Amman 11196 Jordan}
\email{malmheidat@uop.edu.jo}

\begin{abstract}
In this paper an explicit form of generalized Chebyshev Koornwinder's type polynomial of first kind in terms of the Bernstein basis of fixed degree $n$ is provided. Moreover, we investigate generalized Chebyshev Koornwinder's type polynomials of first kind and Bernstein polynomials change of bases.
\end{abstract}      
\maketitle

\section{Introduction}
Approximation is important to many numerical methods, as many arbitrary continuous functions can be approximated by polynomials. On the other hand, polynomials can be characterized in many different bases. Every type of polynomial basis has its strength, where by suitable choice of the basis, numerous problems can be solved and many complications can be removed.

Characterization of the generalized Chebyshev Koornwinder's type polynomials of first and second kind using Bernestin basis were discussed in \cite{alqudah:Characterization, Alqudahn:CharaII} respectively. On the other side, Rababah \cite{rababah4} investigated  the transformation of classical Chebyshev Bernstein polynomial basis.

In this paper,  we generalize the results in \cite{rababah4}, where an explicit closed form of generalized Chebyshev-Koornwinder's type polynomials of first kind (ChK-1) of degree $r\leq n$ expressed in terms of Bernstein basis of fixed degree $n,$ and a closed forms of generalized Chebyshev Koornwinder's polynomials and Bernstein polynomials bases transformations are provided.
\subsection{Bernstein polynomials}
The Bernstein polynomials have been studied thoroughly, there exist many great enduring works on theses polynomials \cite{Farouki4}. 
The $n+1$ Bernstein polynomials $B_{k}^{n}(x)$ of degree $n,$ $x\in [0,1],$ are defined by
$B_{k}^{n}(x)=\frac{n!}{k!(n-k)!} x^{k}(1-x)^{n-k}, \hspace{.05in} k=0,1,\dots,n.$

Bernstein polynomials are the standard basis for the B\'{e}zier  curves and surfaces in Computer Aided Geometric Design (CAGD). Many applications that involve two or more B\'{e}zier curves require all involved B\'{e}zier  curves to have the same degree, so we use the degree elevation defined by \cite{Farouki3}
to write Bernstein polynomials of degree $r$ where $r\leq n$ in terms of the Bernstein polynomials of degree $n$ 
\begin{equation}\label{ber-elv} B_{k}^{r}(x)=\sum_{i=k}^{n-r+k}\frac{\binom{r}{k}\binom{n-r}{i-k}}{\binom{n}{i}}B_{i}^{n}(x),\hspace{.1in}k=0,1,\dots,r.
\end{equation}
Although higher degree B\'{e}zier curves require more time to process, they do have more flexibility for designing shapes, where the basis are known to be optimally stable. So, it would be very helpful to increase the degree of a B\'{e}zier curve without changing its shape  \cite{Farin,Hoschek,Yamaguchi},.

Farouki \cite{Farouki4} surveyed Bernstein basis properties, he described some of the key properties and algorithms associated with the Bernstein basis. Bernstein polynomials remarkable properties \cite{Farouki4} make them significant for the development of B\'{e}zier curves and surfaces in CAGD. 
However, the Bernstein polynomials are not orthogonal and could not be used effectively in the least-squares approximation \cite{Rice}. the calculations performed in finding the least-square approximation polynomial of degree $m$ do not decrease the calculations to obtain the least-squares approximation polynomial of degree $m+1.$
Since then, the method of least square approximation accompanied by orthogonality polynomials has been presented and examined.

In the following, we define least-square approximations of a continuous function $f(x)$ by using polynomials with standard power basis $1,x,x^{2},\dots,x^{n}.$

\begin{definition}
For a continuous function $f(x)$ defined on $[0,1],$ the least square approximation requires finding a least-squares polynomial $p_{n}^{*}(x)=\sum_{k=0}^{n}a_{k}\phi_{k}(x)$
that minimizes the error $E(a_{0},a_{1},\dots,a_{n})=\int_{0}^{1}[f(x)-p_{n}^{*}(x)]^{2}dx.$
\end{definition}

A necessary condition for $E(a_{0},a_{1},\dots,a_{n})$ to have a minimum over all values $a_{0},a_{1},\dots,a_{n},$ is 
$\displaystyle 0=\frac{\partial E}{\partial a_{k}}, k=0,\dots,n.$
Thus, for $i=0,1,\dots,n,$ $a_{i}$ that minimize $\left\|f(x)-\sum_{k=0}^{n}a_{k}\phi_{k}(x)\right\|_{2}$ satisfy 
$$\int_{0}^{1}f(x)\phi_{i}(x)dx=\sum_{k=0}^{n}a_{k}\int_{0}^{1}\phi_{k}(x)\phi_{i}(x)dx,$$
which gives a system of $(n+1)$ equations in $(n+1)$ unknowns: $a_{i},$ $i=0,\dots,n,$ called normal equations. Those $(n+1)$ unknowns of the least-squares polynomial $p_{n}^{*}(x),$  can be found by solving the normal equations.
By choosing $\phi_{i}(x)=x^{i},$ as a basis, then 
$$\int_{0}^{1}f(x)x^{i}dx=\sum_{k=0}^{n}a_{k}\int_{0}^{1}x^{i+k}dx=\sum_{k=0}^{n}\frac{a_{k}}{i+k+1}.$$

The coefficients matrix of the normal equations is Hilbert matrix,  $H_{ij}\equiv(i+j-1)^{-1},$ $i,j=1,\dots,n,$ which has round-off error difficulties and notoriously ill-conditioned for even modest values of $n.$
However, such computations can be made computationally effective by using orthogonal polynomials. Thus,  choosing $\{\phi_{0}(x),\phi_{1}(x),\dots,\phi_{n}(x)\}$ to be orthogonal simplifies the least-squares approximation problem. The coefficients matrix of the normal equations will be diagonal, which simplifies calculations and gives a closed form for $a_{i},i=0,1,\dots,n.$ Moreover, once $p_{n}(x)$ is known, it is only needed to compute $a_{n+1}$ to get $p_{n+1}(x),$ which turns out to be computationally efficient.
to get $p_{n+1}(x),$ which turns out to be computationally efficient. See \cite{Rice} for more details on the least squares approximations.

\subsection{Factorials and double factorials}
The double factorial of an integer $m$ is given by
$(2m-1)!!=(2m-1)(2m-3)(2m-5)\dots(3)(1)$ if $m$ is odd, and $m!!=(m)(m-2)(m-4)\dots(4)(2)$ if $m$ is even, where $0!!=(-1)!!=1.$  
\begin{definition}For an integer $n,$ the double factorial is defined as
\begin{equation}\label{douvle-factrrial} n!!=\left\{\begin{array}{ll}2^{\frac{n}{2}}(\frac{n}{2})! & \text{if $n$ is even} \\\frac{n!}{2^{\frac{n-1}{2}}(\frac{n-1}{2})!} & \text{if $n$ is odd} \end{array}\right..\end{equation}
\end{definition}
From the definition \ref{douvle-factrrial}, and the fact $(1/2)!=\sqrt{\pi}/2$ we can derive the factorial of an integer minus half as 
\begin{equation*}\label{n-half}
\left(n-\frac{1}{2}\right)!=\frac{n!(2n-1)!!}{(2n)!!}\sqrt{\pi}.
\end{equation*}
\subsection{The generalized Chebyshev-I Koornwinder's type polynomials}
For $M,N\geq 0,$ the generalized Chebyshev-I  Koornwinder's type polynomials  \cite{alqudah:Characterization}  of degree $n$
\begin{equation}\label{gen-Chebyshev-I-ccomb}
\mathscr{T}_{n}^{(M,N)}(x)=\frac{(2n-1)!!}{(2n)!!}T_{n}(x)+\sum_{k=0}^{n} \frac{(2k)!\lambda_{k}}{2^{2k}(k!)^{2}} T_{k}(x), \hspace{.1in}n=0,\dots,\infty,
\end{equation}
are orthogonal on the interval $[-1,1]$ with respect to the measure 
$\displaystyle\frac{\sqrt{1-x^{2}}}{\pi}dx+M\delta_{-1}+N\delta_{1},$ where $\delta_{x}$ is a singular Dirac measure,

$T_n(x)$ are the Chebyshev polynomials of the first kind (Chebyshev-I), and
\begin{equation}\label{lmda}\lambda_{k}= \frac{4M}{(2k-3)(k-1)!}+ \frac{4N}{(2k-3)(k-1)!}+\frac{4MN}{(k-1)!(k-2)!}.\end{equation}

The following theorem \cite{alqudah:Characterization} shows how generalized Chebyshev-I  Koornwinder's type polynomials $\mathscr{T}_{r}^{(M,N)}(x)$ of degree $r$ can be written explicitly as a linear combination of the Bernstein polynomial basis of degree $r.$ 

\begin{theorem}\label{gen-jacinBer form}\cite{alqudah:Characterization}(M. AlQudah)
For $M,N\geq 0,$ generalized Chebyshev-I  Koornwinder's type polynomials $\mathscr{T}_{r}^{(M,N)}(x)$ of degree $r$ have the following Bernstein representation,
\begin{equation}\label{gen-jac-inBer-r}
\mathscr{T}_{r}^{(M,N)}(x)=\frac{(2r-1)!!}{(2r)!!}\sum_{i=0}^{r}(-1)^{r-i}\eta_{i,r}B_{i}^{r}(x)+\sum_{k=0}^{r}
\frac{(2k)!\lambda_{k}}{2^{2k}(k!)^{2}}\sum_{j=0}^{k}(-1)^{k-j}\eta_{j,k}B_{j}^{k}(x),
\end{equation}
where $\lambda_{k}$ defined in \eqref{lmda}, $\eta_{i,r}=\frac{\binom{2r}{r}\binom{2r}{2i}}{2^{2r}\binom{r}{i}}, i=0,1,\dots,r,$ and
$\eta_{0,r}=\frac{1}{2^{2r}}\binom{2r}{r}.$\\
Moreover, the coefficients $\eta_{i,r}$ satisfy the recurrence relation
\begin{equation}\label{recur}\eta_{i,r}=\frac{(2r-2i+1)}{(2i-1)}\eta_{i-1,r}, \hspace{.1in}i=1,\dots,r.\end{equation}
\end{theorem}

\section{Change of Bases}
In this section we provide a closed form for the transformation matrix of generalized Chebyshev-I Koornwinder's polynomials basis into Bernstein polynomials basis, and for Bernstein polynomials basis into generalized Chebyshev-I Koornwinder's polynomials basis.

\subsection{Generalized Chebyshev-I to Bernstein}
Rababah \cite{rababah4} provided some results concerning the classical Chebyshev case. In the following theorem we generalize the procedure in \cite{rababah4}. The following theorem will be used to combine the superior performance of the least-squares of the generalized Chebyshev-I polynomials with the geometric insights of the Bernstein polynomials basis.
\begin{theorem}\label{trans}
The entries $M_{i,r}^{n}, i,r=0,1,\dots,n$ of the matrix transformation of the generalized Chebyshev-I polynomial basis into Bernstein polynomial basis of degree $n$ are
given by
\begin{equation}\label{gen-mu}
\begin{aligned}
M_{i,r}^{n}=&\binom{r}{i}^{-1}\frac{(2r)!}{2^{2r}(r!)^{2}}\sum_{k=\max(0,i+r-n)}^{\min(i,r)}(-1)^{r-k}\binom{n-r}{i-k}\binom{r-\frac{1}{2}}{k}\binom{r-\frac{1}{2}}{r-k}\\
&+\sum_{k=0}^{r}\lambda_{k}\binom{k}{i}^{-1}\frac{(2k)!}{2^{2k}(k!)^{2}}\sum_{j=\max(0,i+k-n)}^{\min(i,k)}(-1)^{k-j}\binom{n-k}{i-j}\binom{k-\frac{1}{2}}{j}\binom{k-\frac{1}{2}}{k-j}
.\end{aligned}
\end{equation}
\end{theorem}
\begin{proof}
Write $p_{n}(x)$ as a linear combination of the Bernstein polynomial basis as and as a linear combination of the generalized Chebyshev-I polynomials as follows
$p_{n}(x)=\sum_{r=0}^{n} c_{r}B_{r}^{n}(x)=\sum_{i=0}^{n}d_{i}\mathscr{T}_{i}^{(M,N)}(x).$

We are interested in matrix $M,$ 
which transforms the generalized Chebyshev-I coefficients $\{d_{i}\}_{i=0}^{n}$ into the Bernstein coefficients $\{c_{r}\}_{r=0}^{n},$   
$c_{i}=\sum_{r=0}^{n}M_{i,r}^{n}d_{r}.$

By using Theorem \ref{gen-jacinBer form} to write the generalized Chebyshev-I polynomials \eqref{gen-Chebyshev-I-ccomb} of degree $r\leq n$ in terms of Bernstein polynomial basis of degree $n$ as 
\begin{equation}\label{fun-berns}\mathscr{T}_{r}^{(M,N)}(x)=\sum_{i=0}^{n}N_{r,i}^{n}B_{i}^{n}(x), \hspace{.1in} r=0,1,\dots,n,\end{equation}
where $N_{r,i}^{n}$ the entries of the $(n+1)\times(n+1)$ basis conversion matrix $N.$ Thus, the elements of $c$ can be written in the form 

$c_{i}=\sum_{r=0}^{n}d_{r}N_{r,i}^{n}.$

Now, we need to write each Bernstein polynomial of degree $r$ where $r\leq n$ in terms of Bernstein polynomials of degree $n.$ By substituting the degree elevation \eqref{ber-elv} 
into \eqref{gen-jac-inBer-r} and rearrange the order of summations, we find that the entries of the matrix $N=M^{T}$ for $r=0,1,\dots,n$ are given by
\begin{align*}
N_{r,i}^{n}=&\binom{r}{i}^{-1}\frac{(2r)!}{2^{2r}(r!)^{2}}\sum_{k=\max(0,i+r-n)}^{\min(i,r)}(-1)^{r-k}\binom{n-r}{i-k}\binom{r-\frac{1}{2}}{k}\binom{r-\frac{1}{2}}{r-k}\\
&+\sum_{k=0}^{r}\lambda_{k}\binom{k}{i}^{-1}\frac{(2k)!}{2^{2k}(k!)^{2}}\sum_{j=\max(0,i+k-n)}^{\min(i,k)}(-1)^{k-j}\binom{n-k}{i-j}\binom{k-\frac{1}{2}}{j}\binom{k-\frac{1}{2}}{k-j}
.\end{align*}
\end{proof}

In the following Corollary, we express the generalized Chebyshev-I polynomials $\mathscr{T}_{r}^{(M,N)}(x)$ of degree $r\leq n$ in terms of
the Bernstein basis of fixed degree $n.$

\begin{corollary}\label{gen-rabab2}
The generalized Chebyshev-I polynomials $\mathscr{T}_{0}^{(M,N)}(x),\dots,\mathscr{T}_{n}^{(M,N)}(x)$ of degree less than or equal to $n$ can be expressed in the Bernstein basis of fixed degree $n$ by the following formula
\begin{equation*}\mathscr{T}_{r}^{(M,N)}(x)=\sum_{i=0}^{n}N_{r,i}^{n}B_{i}^{n}(x), \hspace{.1in} r=0,1,\dots,n,\end{equation*}
where
$N_{r,i}^{n}=
\mu_{r,i}^{n}
+\sum_{k=0}^{r}\lambda_{k}
\mu_{k,i}^{n}$
and $\mu_{r,i}^{n}$ defined as
$$\mu_{r,i}^{n}=\frac{(2r)!}{2^{2r}(r!)^{2}}\sum_{k=\max(0,i+r-n)}^{\min(i,r)}
(-1)^{r-k}
\frac{\binom{n-r}{i-k}
\binom{2r}{r}\binom{2r}{2k}}{2^{2r}\binom{n}{i}}.$$
\end{corollary}

\begin{proof}
From equation \eqref{fun-berns} From the proof of Theorem \ref{trans}, any generalized Chebyshev-I polynomials $\mathscr{T}_{r}^{(M,N)}(x)$
 of degree $r\leq n$ can be expressed in terms of Bernstein basis of fixed degree $n$ as 
$\mathscr{T}_{r}^{(M,N)}(x)=\sum_{i=0}^{n}N_{r,i}^{n}B_{i}^{n}(x), \hspace{.1in} r=0,1,\dots,n,$
where the matrix $N$ can be obtained by transposing the entries of the matrix $M$ defined in \eqref{gen-mu}. 
But,
\begin{align*}
\binom{r-\frac{1}{2}}{r-k}\binom{r-\frac{1}{2}}{k}&=\frac{(2r-1)!!}{2^{r}(r-k)!}\frac{2^{k}}{(2k-1)!!}\frac{(2r-1)!!}{2^{r}k!}\frac{2^{r-k}}{(2r-2k-1)!!}\\
&=\frac{1}{2^{r}(r-k)!k!}\frac{(2r-1)!!}{(2k-1)!!}\frac{(2r-1)!!}{(2(r-k)-1)!!}.
\end{align*}
Using the fact $(2r)!=(2r-1)!!2^{r}r!$ we get 
\begin{equation}
\frac{\binom{n-r}{i-k}\binom{r-\frac{1}{2}}{r-k}\binom{r-\frac{1}{2}}{k}}{\binom{n}{i}}=\frac{\binom{n-r}{i-k}\binom{2r}{r}\binom{2r}{2k}}{2^{2r}\binom{n}{i}}.
\end{equation}
So, the entries $N_{r,i}^{n}$ can be rewritten as
\begin{equation*}
N_{r,i}^{n}=
\mu_{r,i}^{n}
+\sum_{k=0}^{r}\lambda_{k}
\mu_{k,i}^{n},
\end{equation*}
where $$\mu_{r,i}^{n}=\frac{(2r)!}{2^{2r}(r!)^{2}}\sum_{k=\max(0,i+r-n)}^{\min(i,r)}
(-1)^{r-k}
\frac{\binom{n-r}{i-k}
\binom{2r}{r}\binom{2r}{2k}}{2^{2r}\binom{n}{i}}.$$
\end{proof}

\subsection{Bernstein to generalized Chebyshev-I} Farouki \cite{Farouki4}  discussed some analytic and geometric properties for Bernstein polynomials. It is worth mentioning the product of two Bernstein polynomials is a Bernstein polynomial and given by
$B_{i}^{n}(x)B_{j}^{m}(x)=\frac{\binom{n}{i}\binom{m}{j}}{\binom{n+m}{i+j}}B_{i+j}^{n+m}(x),$ and the Bernstein polynomials can be integrated easily as $\int_{0}^{1}B_{k}^{n}(x)dx=\frac{1}{n+1},\hspace{.1in} k=0,1,\dots,n.$ In addition, we have the following Lemma that will be used in the proof of next theorem.

\begin{lemma}\label{gen-int-ber-jac}\cite{alqudah:Characterization}(M. AlQudah)
Let $B_{r}^{n}(x)$ be the Bernstein polynomial of degree $n$ and $\mathscr{T}_{i}^{(M,N)}(x)$ be the generalized Chebyshev-I polynomial of degree $i,$ then
for $i,r=0,1,\dots,n$ we have
\begin{equation*}
\begin{aligned}
&\int_{0}^{1}(1-x)^{-\frac{1}{2}}x^{-\frac{1}{2}}B_{r}^{n}(x)\mathscr{T}_{i}^{(M,N)}(x)dx\\
&=\binom{n}{r}\frac{(2i)!}{2^{2i}(i!)^{2}}\sum_{k=0}^{i}(-1)^{i-k}\binom{i-\frac{1}{2}}{k}\binom{i-\frac{1}{2}}{i-k}
B(r+k+\frac{1}{2},n+i-r-k+\frac{1}{2})\\
&+\sum_{d=0}^{i}\lambda_{d}\binom{n}{r}\frac{(2d)!}{2^{2d}(d!)^{2}}\sum_{j=0}^{d}(-1)^{d-j}
\binom{d-\frac{1}{2}}{j}\binom{d-\frac{1}{2}}{d-j}
B(r+j+\frac{1}{2},n+d-r-j+\frac{1}{2})
\end{aligned}
\end{equation*}
where $B(x,y)$ is the Beta function.
\end{lemma}
\begin{proof}
See \cite{alqudah:Characterization} for the proof.
\end{proof}

\begin{theorem}
The entries $M_{i,r}^{n^{-1}},i,r=0,1,\dots,n$ of the matrix of transformation of the Bernstein polynomial basis into the generalized Chebyshev-I polynomial
basis of degree $n$ are given by
\begin{equation}
\begin{aligned}
M_{i,r}^{n^{-1}}=\frac{\Phi_{i}}{(1+\lambda_{i})^{2}}\binom{n}{r}
\left[\frac{2^{2i}(i!)^{2}}{(2i)!}
\Psi_{k,i}^{n,r}
+\left(\frac{2^{2i}(i!)^{2}}{(2i)!}\right)^{2}
\sum_{d=0}^{i}
\frac{(2d)!\lambda_{d}}{2^{2d}(d!)^{2}}
\Psi_{j,d}^{n,r}\right],
\end{aligned}
\end{equation}
where $\lambda_{i}$ defined in \eqref{lmda}, $\Phi_{i}=\left\{\begin{array}{ll} 
2/\pi & \mbox{if } i=0\\
1/\pi & \mbox{if } i\neq 0
\end{array}\right.,$
$$\Psi_{k,i}^{n,r}=\sum_{k=0}^{i}
\frac{(-1)^{i-k}}{2^{2i}}\binom{2i}{i}\binom{2i}{2k}
B(r+k+\frac{1}{2},n+i-r-k+\frac{1}{2}),$$
and $B(x,y)$ is the Beta function.
\end{theorem}
\begin{proof}
To write the Bernstein polynomial basis into generalized Chebyshev-I polynomial basis of degree $n,$ we invert the transformation $c_{i}=\sum_{r=0}^{n}M_{i,r}^{n}d_{r}$ to get
\begin{equation}d=M^{-1}c.\end{equation}
Let $M_{i,r}^{n^{-1}}, N_{i,r}^{n^{-1}},$ $i,r=0,\dots,n$ be the entries of $M^{-1}$ and $N^{-1}$ respectively. The transformation of Bernstein polynomial into
generalized Chebyshev-I polynomial basis of degree $n$ can then be written as
\begin{equation}\label{gen-Ber-Jac}B_{r}^{n}(x)=\sum_{i=0}^{n}N_{r,i}^{n^{-1}}\mathscr{T}_{i}^{(M,N)}(x).\end{equation}

The entries $N_{r,i}^{n^{-1}},$ $i,r=0,1,\dots,n$ can be found by we multiply \eqref{gen-Ber-Jac} by 
$(x-x^{2})^{-\frac{1}{2}}\mathscr{T}_{i}^{(M,N)}(x)$ and integrate
over $[0,1]$ to get
\begin{equation}
\int_{0}^{1}(x-x^{2})^{-\frac{1}{2}}B_{r}^{n}(x)\mathscr{T}_{i}^{(M,N)}(x)dx
=\sum_{i=0}^{n}N_{r,i}^{n^{-1}}\int_{0}^{1}(x-x^{2})^{-\frac{1}{2}}\mathscr{T}_{i}^{(M,N)}(x)\mathscr{T}_{i}^{(M,N)}(x)dx.
\end{equation} 
By using the orthogonality relation 
\begin{equation*}\label{orth-rel}
\int_{0}^{1}(1-x)^{-\frac{1}{2}}x^{-\frac{1}{2}}T_{n}(x)T_{m}(x)dx=\left\{\begin{array}{ll} 
0& \mbox{if } m\neq n \\
\frac{\pi}{2} & \mbox{if } m=n=0\\
\pi & \mbox{if } m=n=1,2,\dots
\end{array}\right.
\end{equation*}
we get

\begin{equation*}
\int_{0}^{1}(x-x^{2})^{-\frac{1}{2}}B_{r}^{n}(x)\mathscr{T}_{i}^{(M,N)}(x)dx
=\left\{\begin{array}{ll} 
\frac{\pi}{2}N_{r,i}^{n^{-1}}\left(\frac{(2i)!}{2^{2i}(i!)^{2}}\right)^{2}(1+\lambda_{i})^{2} & \mbox{if } i=0\\
\pi N_{r,i}^{n^{-1}}\left(\frac{(2i)!}{2^{2i}(i!)^{2}}\right)^{2}(1+\lambda_{i})^{2} & \mbox{if } i\neq 0
\end{array}\right..
\end{equation*} 
Taking into account Theorem \ref{gen-int-ber-jac}, we obtain
\begin{equation*}
\begin{aligned}
N&_{r,i}^{n^{-1}}=\frac{\Phi_{i}\binom{n}{r}}{(1+\lambda_{i})^{2}}
[\frac{2^{2i}(i!)^{2}}{(2i)!}
\sum_{k=0}^{i}
\frac{(-1)^{i-k}}{2^{2i}}\binom{2i}{i}\binom{2i}{2k}
B(r+k+\frac{1}{2},n+i-r-k+\frac{1}{2})\\
&+\left(\frac{2^{2i}(i!)^{2}}{(2i)!}\right)^{2}
\sum_{d=0}^{i}
\frac{(2d)!\lambda_{d}}{2^{2d}(d!)^{2}}
\sum_{j=0}^{d}
\frac{(-1)^{d-j}}{2^{2d}}\binom{2d}{d}\binom{2d}{2j}
B(r+j+\frac{1}{2},n+d-r-j+\frac{1}{2})].
\end{aligned}
\end{equation*}
By reordering the terms we have
\begin{equation}
\begin{aligned}
N_{r,i}^{n^{-1}}=\frac{\Phi_{i}}{(1+\lambda_{i})^{2}}\binom{n}{r}
\left[\frac{2^{2i}(i!)^{2}}{(2i)!}
\Psi_{k,i}^{n,r}
+\left(\frac{2^{2i}(i!)^{2}}{(2i)!}\right)^{2}
\sum_{d=0}^{i}
\frac{(2d)!\lambda_{d}}{2^{2d}(d!)^{2}}
\Psi_{j,d}^{n,r}\right],
\end{aligned}
\end{equation}
where $\Phi_{i}=\left\{\begin{array}{ll} 
2/\pi & \mbox{if } i=0\\
1/\pi & \mbox{if } i\neq 0
\end{array}\right.,$ $\lambda_{i}$ defined in \eqref{lmda},
$$\Psi_{k,i}^{n,r}=\sum_{k=0}^{i}
\frac{(-1)^{i-k}}{2^{2i}}\binom{2i}{i}\binom{2i}{2k}
B(r+k+\frac{1}{2},n+i-r-k+\frac{1}{2}),$$
and $B(x,y)$ is the Beta function.
The entries of $M^{-1}$ are obtained by transposition of $N^{-1}.$ 
\end{proof}


\end{document}